\documentclass{article}

\usepackage{amsmath}
\usepackage{amsthm}
\usepackage{url}
\usepackage{graphicx}
\usepackage{fullpage}

\newtheorem{thm}{Theorem}
\newtheorem{biog}{}
\newtheorem{affil}{}
\newtheorem{lemma}{Lemma}
\newtheorem{prop}{Proposition}
\theoremstyle{definition}
\newtheorem{definition}{Definition}
\newtheorem{remark}{Remark}

\begin{document}

\title{How Many Real  Attractive Fixed Points Can A Polynomial Have?}

\author{Terence Coelho\thanks{Undergraduate Math and CS Major,  Rutgers University (terencejcoelho96@gmail.com)} ~and~ Bahman Kalantari\thanks{Professor of Computer Science, Rutgers University (kalantari@cs.rutgers.edu)}}
\date{}
\maketitle
\begin{abstract}
We prove a complex polynomial of degree $n$ has at most $\lceil n/2 \rceil$  attractive fixed points lying on a line. We also consider  the general case.
\end{abstract}



\section{Introduction.} While the notion of roots of  a quadratic is rudimentary in K-12 math, that of its {\it fixed points} is uncommon.  This is surprising because the relevance of the fixed points of a quadratic can be demonstrated easily via iterative methods for approximation of such numbers as $\sqrt{2}$, when the quadratic formula offers no remedy. In fact, in the context of fixed  points one can also give a formal definition of the derivative of a quadratic and see its application in action. More precisely,  if $\theta$ is a zero of a quadratic $q(x)$, then it is a fixed points of  $p(x)=cq(x)+x$, where $c$ is any nonzero constant.  A fixed point  $\theta$ of  $p(x)=a_2x^2+a_1x+a_0$ is {\it attractive} if its {\it multiplier}, $p'(\theta)=2a_2 \theta + a_1$, has absolute value (or modulus in case of complex multiplier) less than one. It is straightforward to show the {\it fixed point iteration} $x_{k+1}=p(x_k)$ converges to $\theta$ starting with any {\it seed} $x_0$ that is sufficiently close to $\theta$.
Specifically,  $p(x_k)- p(\theta)= (x_k- \theta) (p'(x_k)+p'(\theta))/2$.  Thus, when $x_k$ is close to $\theta$, $|x_{k+1}- \theta|$ is approximately  $|p'(\theta)| \cdot |x_k-\theta|$.  Such analysis can  also help in the development of Newton's method and the study of its rate of convergence.  These observations on a quadratic extend to an arbitrary degree polynomial $p(x)$ with the only modification that $p(x_k)- p(\theta)=(x_k-\theta)r(x_k)$, where $r(x)$ is a polynomial that can be shown to satisfy $r(\theta)= p'(\theta)$.  Thus the local convergence of the fixed point iteration to an attractive fixed point follows for any degree polynomial, with real of or complex coefficients.

A natural questions arises: Can all the fixed points of a complex polynomial be attractive?  Anyone familiar with Julia sets  of iterations of $z^2+c$, popularized by Mandelbrot \cite{Man} and the famous set that bears his name, must have noticed that not both fixed points can be attractive.  Surprisingly, despite the existence of many books on fractals, this aspect of iterations of a polynomial does not seem to be emphasized.  The study of this question is  one of the most basic attempts in the understanding of the dynamics of iterations of polynomials and rational functions (see e.g. Beardon \cite{Bea}, Milnor \cite{Mil}), as well as in the study of polynomial root-finding methods,  their fractal behavior and visualization of iterative methods (see e.g. Devaney \cite{Dev}, Mandelbrot \cite{Man}, and \cite{Kal}).

In this note we prove at most $\lceil n/2 \rceil$  attractive fixed points of a complex polynomial of degree $n$ can lie on a  line. In particular, this bounds the number of real attractive fixed points. For each $n >1$, we exhibit a polynomial of degree $n$  with $n-1$  attractive fixed points. Via a known result but a nontrivial proof, $n-1$ is also the maximum number of attractive fixed points for any polynomial of degree $n >1$. In fact we conjecture a stronger result ({\it real unattractive fixed point conjecture}): Any polynomial of degree $n>1$ has a fixed point where the real part of its multiplier is at least one.  We prove it for $n=2,3$.

\section{Bound On Attractive Fixed Points}

\begin{thm} \label{thm1} If $p(z)$ is a complex polynomial having $n$  attractive fixed points that lie on a line, then its degree is at least $2n-1$.
\end{thm}

\begin{proof} Let the $n$ distinct attractive fixed points be $z_1, \dots, z_n$ and $\alpha_i = p'(z_i)$, $\forall i$.  Consider the Hermite interpolating polynomial $H(z)$,  i.e. the least degree polynomial  satisfying  $H(z_i)=z_i$,  $H'(z_i)=\alpha_i$,  $\forall i=1, \dots, n$.
Its degree is at most $2n-1$ and it can be written explicitly via the {\it divided differences}, see \cite{Bur}.  We prove the degree of $H(z)$ is $2n-1$, and the degree of $p(z)$ is at least $2n-1$.

Call a list of  indices $i_1, i_2, \dots, i_k$ with $i_j \in \{1, \cdots, n\}$  {\it valid} if any two identical indices appear as adjacent elements.  The divided differences are defined as follows: Set $f[z_i]=p(z_i)=z_i$ and  $f[z_i,z_i]=p'(z_i)$. Recursively define
\begin{equation} \label{eq0}
f[z_{i_1}, \dots, z_{i_k}]=({f[z_{i_2}, \dots, z_{i_k}] - f[z_{i_1}, \dots, z_{i_{k-1}}]})/({z_{i_k}-z_{i_1}}).
\end{equation}

Thus for $i \not = j$, $f[z_i, z_j]=1$, $f[z_i,z_i,z_j]= (\alpha_i-1)/(z_i-z_j)$, $f[z_i,z_i,z_j,z_j]= ( \alpha_i+\alpha_j -2)/(z_i-z_j)^2$, and $f[z_i,z_i, z_j,z_j, z_k]=$ $(\alpha_i-1)/{(z_i-z_j)^2(z_i-z_k)}+$ $ (\alpha_j-1)/{(z_i-z_j)^2(z_j-z_k)}$.
The Hermite interpolating polynomial is then
$$H(z)= f[z_1]+f[z_1,z_1](z-z_1)+ f[z_1,z_1,z_2](z-z_1)^2+ \cdots+$$
$$f[z_1,z_1,\dots, z_{n-1},z_{n-1}, z_n]\prod_{i=1}^{n-1}(z-z_i)^2+ f[z_1,z_1, \dots, z_n,z_n] (z-z_n) \prod_{i=1}^{n-1}(z-z_i)^2.$$
We will prove  $f[z_1,z_1, \dots, z_n,z_n]$, the coefficient of $z^{2n}-1$, is nonzero.  In doing so we prove by induction that for each $k \in \{2, \dots, n\}$,
\begin{equation} \label{eq1}
f[z_1,z_1, \dots, z_k,z_k]= \sum_{i=1}^k \frac{(\alpha_i-1)}{\pi^k_i}, \quad \text{where} \quad  \pi^k_i=\prod_{j=1, j \not =i}^k (z_i-z_j)^2,
\end{equation}
and for each  $k \in \{2, \dots, n-1\}$,
\begin{equation} \label{eq2}
f[z_1,z_2,z_2, \dots, z_k,z_k, z_{k+1}]= \sum_{i=2}^k \frac{(\alpha_i-1) (z_i-z_1)(z_i-z_{k+1})}{\pi^{k+1}_i}.
\end{equation}

Note that (\ref{eq1}) and (\ref{eq2}) are true for $k=1,2$.
Assume (\ref{eq1}) and (\ref{eq2}) are true for $k$. We will prove they are true for $k+1$.
We first prove (\ref{eq1}). From the definition of divided differences (\ref{eq0}), (\ref{eq1}), (\ref{eq2}), and  that $\pi^{k+1}_i=\pi^k_i(z_i-z_{k+1})^2$, we get
$$
f[z_1, z_1, \dots, z_k,z_k, z_{k+1}]=  \sum_{i=2}^k \frac{(\alpha_i-1) (z_i-z_1)(z_i-z_{k+1})}{(z_{k+1}-z_1)\pi^{k+1}_i} -
$$
\begin{equation} \label{eq3}
\sum_{i=1}^k \frac{(\alpha_i-1) (z_i-z_{k+1})^2}{(z_{k+1}-z_1)\pi^{k+1}_i}=\sum_{i=1}^k \frac{(\alpha_i-1) (z_i-z_{k+1})}{ \pi^{k+1}_i}.
\end{equation}
From  (\ref{eq3}), symmetry and  invariance under valid permutations, we also get
\begin{equation} \label{eq4}
f[z_1, z_2,z_2,  \dots,z_{k+1}, z_{k+1}]= \sum_{i=2}^{k+1} {(\alpha_i-1)(z_i-z_1)}/{\pi^{k+1}_i}.
\end{equation}
Subtracting the left-hand-side of (\ref{eq3}) from the corresponding side of (\ref{eq4}), dividing by $(z_{k+1}-z_1)$ and simplifying the right-hand-side we get (\ref{eq1}) for $k+1$.  Next we prove (\ref{eq2}) for $k+1$.  From (\ref{eq3}) we may also write,

\begin{equation} \label{eq5}
f[z_2,z_2,  \dots,z_{k+1}, z_{k+1}, z_{k+2}]= \sum_{i=2}^{k+1} {(\alpha_i-1)(z_i- z_{k+2}) }/{\pi^{k+1}_i}.
\end{equation}
Subtracting  (\ref{eq4}) from (\ref{eq5}), dividing by $(z_{k+2}-z_1)$ and simplifying, we get (\ref{eq2}) for $k+1$. Thus  (\ref{eq1}) holds for $k=n$.  Suppose for all $i=1, \dots, n$,  $z_i=x_i + y_i \sqrt{-1}$ lies on the same line, say $y=mx+b$. Then $z_i-z_j=(x_i-x_j)(1+m \sqrt{-1})$.  Then from (\ref{eq1}) it follows that $\pi^n_i = r_i c$, where $r_i$ is some positive real number and $c=(1+m \sqrt{-1})^{n-1}$.  Also note that $(\alpha_i-1)$ has a negative real part. It follows that  $f[z_1,z_1, \dots, z_n,z_n]$ is nonzero. This also holds when $z_i$'s all lie on the $y$-axis.

Next we prove that the degree of $p(z)$ is at least $2n-1$.  Suppose otherwise. Then on the one hand $q(z)=H(z)-p(z)$ is a polynomial of degree $2n-1$. On the other hand, since $q(z_i)=q'(z_i)=0$,  $(z-z_i)^2$ must divide $q(z)$ for each $i=1, \dots, n$.  This implies the degree of $q(z)$ is at least $2n$, a contradiction.
\end{proof}

\section{Real Unattractive Fixed Point Conjecture}
Theorem \ref{thm1} does not hold when the $n$ attractive fixed points are arbitrary. Take as an example $p(z)=(-z^{n+1}+(n+1)z)/n$.  The attractive fixed points are the $n$-th roots of unity. The question arises, can all fixed points of a polynomial be attractive?   The answer is negative.  This can be proved in a nontrivial way as follows.  By the fundamental theorem of algebra $p(z)$ has $n-1$ critical points.  On the other hand the basin of attraction of each  attractive fixed point of a rational map must contain a critical point (see e.g. Theorem 5.32 in \cite{Kal}, or \cite{Bea}). Thus not all fixed points of $p(z)$ can be attractive. In fact we make a
 conjecture ({\it real unattractive fixed point conjecture}): A complex polynomial of degree $n >1$ has at least one fixed point  where the real part of its multiplier is at least one.  We prove this for $n=2,3$.  The proof is trivial for $n=2$:  We must have $p(z)=c(z-z_1)(z-z_2)+z$ for some nonzero constant $c$.  Then $p'(z_1)=c(z_1-z_2)+1$ and $p'(z_2)=c(z_2-z_1)+1$.  Since the real part of $c(z_1-z_2)$ and $c(z_2-z_1)$ have  oppositive signs the proof is complete. Next we consider the case where $n=3$.  We have $p(z)=c(z-z_1)(z-z_2)(z-z_3)+z$, where $c$ is a nonzero constant.  Let $a$ be a solution to $a^2=c$.  Let $\alpha_1=a(z_1-z_2)=a_1+ib_1$,  $\alpha_2=a(z_3-z_1)=a_2+ib_2$, and $\alpha_3=a(z_2-z_3)=a_3+ib_3$.  Then $\alpha_1+\alpha_2+\alpha_3=0$.
Note that the multipliers are then $\lambda_1=1- \alpha_1 \alpha_2$, $\lambda_2= 1- \alpha_2 \alpha_3$, and $\lambda_3= 1- \alpha_3 \alpha_1$.  There must exist two indices $i,j$ such that $b_ib_j \geq 0$. Assume without loss of generality $b_1b_2 \geq 0$. To prove our claim we show that the real part of at least one of $- \alpha_i \alpha_j$ is nonnegative for some $i,j$, $i \not =j$.  Otherwise,  we must have $a_1a_2 > b_1b_2$, and substituting for $\alpha_3=-(\alpha_1+\alpha_2)$, we also get,
$a_1^2-b_1^2+a_1a_2-b_1b_2 <0$,  $a_2^2-b_2^2+a_1a_2-b_1b_2 <0$.
From the three strict inequalities we get $a_1^2 < b_1^2$ and $a_2^2 < b_2^2$. But these imply $a_1^2a_2^2 <  b_1^2b_2^2$, contradicting $a_1a_2 > b_1b_2 \geq 0$. Hence the proof of claim for $n=3$.\\

\section*{Final Remarks} A complex polynomial of degree $n > 1$ can have at most $n-1$ attractive fixed points. However, here we proved it can have at most
$\lceil n/2 \rceil$ attractive fixed points lying on a line.  The question may arise if this property of polynomials can also be proved from existing results in dynamical systems.  This is not known to us, however even if possible we doubt such a proof can be established via the elementary technique given here.

In this article we have also  introduced the ``real unattractive fixed point conjecture'' and proved it for $n=2,3$. We feel it would be interesting to prove this for all $n$, and if so in an elementary fashion.

Finally, the notion of fixed point of a polynomial can be introduced in basic math courses in high school and college, helping to promote algorithmic methods for solving polynomial equations, concepts in discrete dynamical systems, as well as visualization of iterative techniques (see \cite{Kal}) which  gives rise to spectacular fractal  and non-fractal images.  These in turn will help promotes novel applications of polynomials.

\end{document}